\documentclass[12pt]{amsart}
\usepackage{amscd, amssymb,amsfonts, amsthm, mathtools, gastex,yhmath}
\usepackage{euscript, MnSymbol}
\usepackage{pdfsync}
\usepackage{enumerate}
\usepackage{epsfig,graphicx,graphics} 
\usepackage{pdfsync}

 \newtheorem{theorem}{Theorem}

    \newtheorem{lemma}[theorem]{Lemma}
\newtheorem{proposition}[theorem]{Proposition}
\newtheorem{corollary}[theorem]{Corollary}
\newtheorem{definition}[theorem]{Definition}
\newtheorem{definitions}[theorem]{Definitions}

\newtheorem{remark}[theorem]{Remark}
\newtheorem{remarks}[theorem]{Remarks}\newtheorem{problem}[theorem]{{\sc Problem}}

\newtheorem{conjecture}[theorem]{{\sc Conjecture}}

\newtheorem{claim}{\noindent {\bf Claim}}[section]

\newcommand{\N}{\mathbb{N}}

\newcommand{\Q}{\mathbb{Q}}

\let\rho=\varrho
\newcommand{\cf}{\mathop{\rm cf}}

\DeclareMathOperator{\res}{res}
\DeclareMathOperator{\NC}{NC}

\DeclareMathOperator{\Sup}{Sup}
\DeclareMathOperator{\Spec}{Spec}

\DeclareMathOperator{\Min}{Min}
\DeclareMathOperator{\age}{Age}
\title{J\'onsson posets}
\author[R.Assous]{Roland Assous}
\address{Universit\'e Claude-Bernard  Lyon1,
43, Bd. du 11 Novembre 1918,
69622 Villeurbanne, France} 
\email{roland.assous@gmail.com}
\author[M.Pouzet] {Maurice Pouzet}
\address{ICJ, UMR 5208, Universit\'e Claude-Bernard  Lyon1,
43, Bd. du 11 Novembre 1918,
69622 Villeurbanne, France et Department of Mathematics and Statistics, The University of Calgary, Calgary, Alberta, Canada}
\email{pouzet@univ-lyon1.fr}

\date{\today}

\begin{document}

\maketitle

To the memory of Bjarni J\'onsson

\begin{abstract}
According to Kearnes and Oman (2013), an ordered set $P$ is \emph{J\'onsson} if it is infinite and the cardinality of every proper initial segment of $P$ is strictly less than the cardinaliy of $P$. We examine the structure  of J\'onsson posets. 
\end{abstract}
\section{Introduction}
An ordered set $P$ is \emph{J\'onsson} if it is infinite and the cardinality of every proper initial segment of $P$ is strictly less than the cardinaliy of $P$. This notion is due to Kearnes and Oman (2013).
J\'onsson posets, notably the uncountable one, appear  in the  study of algebras with the J\'onnson property (algebras for which proper subalgebras have a cardinality strictly smaller than the algebra). The study of these algebras  are the motivation of the paper  by Kearnes and Oman \cite{kearnes} where the  reader will find detailed information. Countable J\'onsson posets  occur naturally   at the  interface of  the theory of relations and symbolic dynamic 
as Theorem \ref{theo-minimal type} and Theorem \ref {indecomposable} below illustrate. They were considered in the thesis of the second author \cite{pouzet-these}, without bearing this name, and characterized in \cite{pouzet-sauer}  under the name of \emph{minimal posets}.

This  characterization involves first the notion of \emph{well-quasi-order}(w.q.o. for short) -- a poset $P$ is w.q.o. if it is \emph{well founded}, that is every non-empty subset $A$ of $P$ has at least a minimal element,  and  $P$ has no infinite antichain --, next, the decomposition of a well founded poset into \emph{levels}  (for each ordinal $\alpha$, the $\alpha$-th \emph{level} is defined by setting $P_{\alpha}:= \Min (P\setminus \bigcup_{\beta<\alpha} P_{\beta})$ so that $P_{0}$ is the set $\Min (P)$ of minimal elements of $P$; each element  $x\in P_{\alpha}$ has \emph{height} $\alpha$, denoted by  $\mathbf h_P(x)$; the \emph{height} of $P$, denoted by $\mathbf h(P)$, is the least ordinal $\alpha$  such that $P_{\alpha}= \emptyset$) and, finally, the notion of \emph{ideal} of a poset (every non-empty initial segment which is up-directed). 


The following result reproduced from \cite{pouzet-sauer} gives  a full description of countable J\'onsson posets. 

\begin{proposition}\label{label:prop:minimaltype}Let $P$ be an
infinite  poset. Then, the following properties are  equivalent:
\begin{enumerate}[{(i)}]
\item $P$ is w.q.o. and all ideals distinct from $P$ are principal;
\item $P$ has no infinite antichain and all ideals distinct from $P$ are finite;
\item Every proper initial segment of $P$ is finite;
\item Every linear extension of $P$ has order type $\omega$ (the order type of chain $\N$ of non negative integers);
\item $P$ is level-finite, of height  $\omega$, and 
 for each $n<\omega$ there is  $m<\omega$ such that each element of 
height at most $n$ is below every element of height at least $m$;
\item $P$ embeds none of the following posets: an infinite antichain; a chain of order type
$\omega ^{*}$ (the dual of the chain of non-negative integer); a chain of order type $\omega +1$; the direct sum $\omega\oplus 1$ of a chain of
order type $\omega$ and a one element chain.
 \end{enumerate}
\end {proposition}
The equivalence between item $(iii)$, $(iv)$ and $(v)$ was given  in \cite{pouzet-these}. One proves 
 $$(i)\Rightarrow (ii)\Rightarrow (iii)\Rightarrow (iv)\Rightarrow (v)\Rightarrow (vi)\Rightarrow (i)$$ using straightforward arguments.

Posets as above are said \emph{minimal}  or  having \emph{minimal type}.\\

A new characterization involving semiorders (posets which do not embed $2 \oplus 2$, the  direct sum  of two $2$-element chains,  nor $3 \oplus 1 $, the direct sum of a $3$-element chain and a $1$-element chain)  is given in Subsection \ref{subsection semiorders}.  In order to present this characterization, we  say that an order $L$, considered as a set of ordered pairs, is \emph{between} two orders $A$ and $B$ on the same set if $A\subseteq  L\subseteq B$.  We prove:

\begin{proposition} \label{prop:semiorder}
A poset has minimal type iff the order is  between a semiorder with no maximal element and a linear order of order type $\omega$.  
\end{proposition}

The requirement that there is no maximal element is essential (otherwise an antichain will satisfies the stated conditions). 

An easy way of obtaining posets with minimal type is given by the following corollary  of Proposition \ref{label:prop:minimaltype}:
\begin {corollary}\label{cor:dim<n}
Let $n$ be a non-negative integer and $P$ be  a poset.  The order on $P$ is the intersection of $n$ linear
orders of  order type $\omega$  if and only if $P$ is the intersection of $n$ linear orders  and $P$ has minimal type.
\end {corollary}

We illustrate first the role of minimal posets  in  the theory of relations. A binary relational structure is a pair $R:= (V, (\rho_{i})_{i\in I})$ made of a set $V$ and a family of  binary relations $\rho_i$ on $V$. A subset $A$ of $V$ is an \emph{interval} of $R$ if for every $x,x'\in A$, $y\in V\setminus A$, $i\in I $, $x\rho_i y$ iff $x'\rho_i y$ and $y\rho_i x$ iff $y\rho_i x'$. The structure $R$ is \emph{indecomposable} if its only intervals are the empty set, the singletons and the whole set $V$. Fix a set $I$;  relational  structures of the form $R:= (V, (\rho_{i})_{i\in I})$ will have  \emph{type} $I$. If $R'$ is an other structure of type $I$, then $R$ is \emph{embeddable} into $R'$ and we set $R\leq R'$ if $R$ is  isomorphic to an induced substructure of $R'$. The \emph{age} of $R$ is the set $\age (R)$ of isomorphic types of finite structures which are embeddable to $R$, these finite structures being considered up to isomorphy.  Let $Ind(\Omega(I))$ be the set of finite indecomposable structures of type $I$.  A subset $\mathcal D $ of $Ind(\Omega(I))$ is \emph{hereditary} if $R\in Ind(\Omega(I))$, $R\leq R'\in \mathcal D$ imply $R\in \mathcal D$. 

The following result was obtained by D.Oudrar in her thesis \cite{oudrarthese}
in collaboration with the second author. 

\begin{theorem} \label{indecomposable} An  infinite  hereditary  subset $\mathcal D$ of  $Ind(\Omega(I))$ contains a hereditary  subset $\mathcal D'$ having minimal type iff it contains only finitely many members of size $1$ or $2$. If $\mathcal D'$ has minimal type then $\downarrow \mathcal D'$ (the set of isomorphic types of $S$ which are embeddable into some $S'\in \mathcal D$) is the age of an  infinite indecomposable structure; this age is well-quasi-ordered by embeddability. 
\end{theorem}
 For example, if $R$ is the infinite path on the set of non negative integers, the age of $R$ consists of direct sums of finite paths. Those finite paths are the indecomposable members of $\age(R)$. They form  a set having  minimal types. Uncountably many sets of binary relations having  minimal type are given in Chapter 5 of \cite{oudrarthese}. Recent results of \cite{chudnovski, kim, malliaris-terry} suggest that a complete characterization is attainable. 

Let us turn to  symbolic dynamic. Let $S: A^{\omega}\rightarrow A^
 {\omega}$ be the shift operator on the set $A^{\omega}$ of infinite sequences $s:= (s_{n})_{n<\omega}$ of elements of a finite set $A$ (that is $S(s):= (s_{n+1})_{n<\omega}$).  A subset $F$ of $A^{\omega}$ is {\it shift-invariant} if $S(F)\subseteq F$ where $S(F):=\{S(x):x\in F\}$. It is \emph{minimal} if it is non-empty,  compact,  shift-invariant and if no proper  subset has the same properties (cf \cite{allouche}). As it is well-known, every compact (non-empty) invariant subset  contains a minimal one.  An infinite word $u$ is called \emph{uniformly recurrent} if  the adherence of the set of its translates via the shift operator is minimal (hence, all $u$ belonging to a minimal set are uniformly recurrent.  For example, infinite Sturmian words are uniformly recurrent and the set $S_{\alpha}$ of  infinite  Sturmian words with slope $\alpha$ is minimal (see Chapter 6 of \cite{fogg}).    To a compact invariant subset $F$ we may associate the set $\mathcal A(F)$ of finite sequences $s:= (s_{0}, \dots, s_{n-1})$ such that $s$ is an initial segment  of some member of $F$. Looking as these sequences as words, we may order $\mathcal A(F)$ by the factor ordering: a sequence $s$ being a {\it factor} of a sequence $t$ if  $s$ can be obtained from $t$ by deleting an initial segment  and a final segment of $ t$. 

We have then: 
\begin{theorem} \label{theo-minimal type}
 $\mathcal A(F)$ has minimal type if and only if $F$ is  minimal.
\end{theorem}

This result about words can be viewed  as a special instance of \mbox{Theorem \ref{indecomposable}}. Indeed, to an infinite sequence $s\in A^{\omega}$ we may associate the relational structure $R_s:= (\N, c, (u_i)_{i\in A})$, 
where $c$ is the binary relation on $\N$ defined by $n c m$ if $m=n+1$ and $u_i$ is the unary relation such that $u_i n$ iff $s(n)=i$.    Let  $Ind(\Omega)$ be the collection of finite indecomposable relational structures with the same type as $R$. Then,  $s$ is uniformly recurrent iff $\age(R_s)\cap Ind(\Omega)$ has minimal type. 

For more about the combinatorial aspects of symbolic dynamic, see \cite{allouche, fogg, lothaire}. 

Posets of minimal type are related to a notion of Jaco graph introduced by Johann Kok in 2014 and studied by him and his collaborators \cite{kok}. A directed graph $G$ is  
a \emph{Jaco graph} if its  vertex set $V(G)$ is the set of positive  integers and there is a   nondecreasing sequence $(a_n)_{n\in  \N^{*}}$ of positive integers such that  a pair  $(n, m)$  forms an arc of $G$ iff    $n< m\leq a_n+n$.   Given a Jaco graph $G$, let $G^d$ be the \emph{directed  complement} of $G$, that is the graph made of directed pairs $(n, m)$ such that $n<m$ but  $(n, m)$ is not an arc of $G$. All pairs $(n, m)$ such that $a_n+n<m$ belong to  this graph, hence they  define  a strict order (i.e.,  a irreflexive and transitive relation) on $\N^*$. This ordered set has minimal type.  In fact,  an order is minimal iff it can be labelled in such a way that it extends the directed complement of a Jaco graph.

Concerning the structure of J\'onsson posets, we note that countable J\'onsson chains are isomorphic to the chain $\omega$.  J\'onsson chains which are well founded are isomorphic to initial ordinals, alias cardinals. But there are uncountable  J\'onsson chains which are not well founded. For an example, the chain $\omega^*\cdot \omega_1$, lexicographical sum along the chain $\omega_1$ of copies of $\omega^*$ (the dual of $\omega$), is J\'onsson but not well founded. Next, uncountable J\'onsson posets may contain  infinite antichains as for an example $\Delta_{\aleph_0}\cdot \omega_1$, the sum along the chain $\omega_1$ of copies of the countable antichain $\Delta_{\aleph_0}$.
 
 Still, uncountable J\'onsson posets, w.q.o. or not,  retain several properties of minimal posets. In this paper, we  give several characterizations  of  J\'onsson posets (e.g. Theorem \ref{labtheo:cof}). We  give  a description of those  whose cardinality is regular (see Theorem \ref{theo:maintheo}): we observe  that a poset $P$ of regular cardinality $\kappa>\aleph_0$ is J\'onsson if and only if it decomposes into a lexicographic sum $\sum_{\alpha\in C }P_{\alpha}$ where $C$ is  a chain of type $\kappa$ and   every $P_{\alpha}$ is a non empty poset of cardinality $\kappa_{\alpha}$ stricly less than $\kappa$. Introducing \emph{pure} posets, we extend this characterization to posets of singular cofinality and give an extension of  Proposition \ref{prop:semiorder} (Theorem \ref{thm:semiorder}). 

 The case of singular cardinality is more subtle, especially when the cofinality is countable.  A description of J\'onsson w.q.o. posets of singular cardinality seems to be  difficult in regard of the following example. 
 
Let $\kappa$ be a singular cardinal with cofinality $\nu$; let $(\kappa_{\alpha})_{\alpha< \nu}$ be  a sequence of cardinals cofinal in $\kappa$; let $[\nu]^2:= \{(\alpha, \beta): \alpha< \beta<\nu\}$ be ordered componentwise and   for each $u:= (\alpha,  \beta)\in [\nu]^2$, let $P_u$  be any w.q.o. of cardinality $\kappa_{\alpha}$. Then $P:= \sum_{u\in[\nu]^2}P_u$  is w.q.o. and  J\'onsson (but not pure, cf. definition \ref{def:pure}).

Our motivation for looking at w.q.o. J\'onsson posets is a beautiful conjecture of Abraham, Bonnet and  Kubis \cite{abraham-bonnet-kubis} relating the notion of w.q.o. and the stronger notion of better-quasi-order (b.q.o.) invented by Nash-Williams \cite{nashwilliams1, nashwilliams2}. Up to now, our attempt has been unsuccesful. 

%

Part of this work is based on an unpublished  study of spectra of  posets \cite{assous-pouzet-Nov86}. 
\section{Terminology, notation} Our terminology is based on \cite{fraisse} and \cite{jech}. We denote cardinal numbers by greek letters, like $\kappa$, $\lambda$ , $\mu$ and by $\vert X\vert$ the cardinality of set $X$. We identify a binary relation $\rho$ on a set $X$ with a set of ordered pairs and we set $x\rho y$ if $(x,y)\in \rho$.  We say that $\rho$ is a \emph{quasi-order} (or a preorder) if it is reflexive and transitive, in which case we say that the set $X$ is \emph{quasi-ordered}. Let $\leq $ be a quasi-order on $X$;  we say that $x$ and $y$ are \emph{comparable} if $x\leq y$ or $y\leq x$ and we set $x\sim y$ (despite the fact that this relation is not an equivalence relation); otherwise we say that $x$ and $y$ are \emph{incomparable} and we set $x\nsim y$.   The relation $\equiv$ defined by $x\equiv y$ if $x \leq y$ and $y\leq x$ is an equivalence relation, whereas the relation $<$ defined by $x<y$ if $x\leq y$ and $y\not \leq x$ is  a  \emph{strict order}. The   relation  $<$ is transitive and irreflexive; in fact, every transitive and irreflexive relation is a strict order.  The relation $\leq$ is \emph{total} if for every two elements $a, b\in X$ either $a\leq b$ or $b\leq a$ holds.  The quasi-order  $\leq$   is a \emph{partial order} and the pair  $P:=(V, \leq)$ is \emph{partially ordered} (poset for short) if $\leq$ it is antisymmetric. 
 A set of pairwise incomparable elements of  a poset is called an \emph{antichain}. A \emph{chain} is a totally ordered set. Let $P$ be  a poset. 
A \emph{strengthening} (also called an \emph{extension}) of $P$ is any order $\leq'$ on $X$ containing the original order. A \emph{linear extension} of $P$ is any linear order containing  this order. 
When we mention elements or subsets of a poset, we usually mean  elements or subsets of its vertex set. Sometimes, we use the same terminology for orders and posets. We feel free to say J\'onsson order as well as  J\'onsson poset.  A subset 
$A$ of a poset $P$ is a {\it final (resp.  initial) segment\/ of  $P$} if whenever
$a\in A$ and $x\geq a$ (resp. $ x\leq a$), then $ x\in A$.  For $A\subseteq P$ we set  $\uparrow\! A:= \{x\in P:\exists a\in 
A(x\geq
a)\}$, this is the \emph{final segment generated by $A$}; we denotes by $\downarrow\! A$  the corresponding initial segment
generated by $A$;
$\uparrow\! a$ and $\downarrow\! a$ abbreviate $\uparrow\! \{a\}$ and
$\downarrow\! \{a\}$. We denote by $\mathcal I(P)$, resp. $\mathcal F(P)$, the set of initial, resp. final, segments of $P$. 
%
A subset $A\subseteq  P$ of a poset $P$ is \emph{cofinal} in $P$ if every $x\in P$ is majorized by some $y
\in  A$, and the {\it cofinality} of $P$, denoted by $\cf(P)$, is 
the
least cardinal $\nu$ such that $P$ contains a cofinal subset of 
cardinality
$\nu$. If $\kappa$ is a cardinal, the \emph{cofinality} of $\kappa$, denoted by $\cf(\kappa)$ is the cofinality of $\kappa$ viewed as an initial ordinal. A cardinal is \emph{regular} if it is equal to its cofinality, otherwise it is \emph{singular}.  A basic property of linearly ordered sets, observed by 
Hausdorff
(see \cite{jech}), is that they contain well ordered cofinal subsets;  it
follows that either they have a maximum element or their cofinality is 
an infinite regular
cardinal.  For an arbitrary poset  $P$, the corresponding fact is that
it contains a well founded cofinal subset, but in this case,  the
cofinality, $\cf(P)$, might be a singular cardinal.

 \section{Characterizations and descriptions of J\'onsson posets}

\subsection{Strengthening of J\'onsson posets}

We start with the following observation:
\begin{lemma}\label{lem:observation} A poset $P$ is J\'onsson iff $P$ is infinite and for every $x\in P$, \mbox{$\vert P\setminus \uparrow x\vert< \vert P\vert$}.
\end{lemma} 

The proof is immediate:  every proper initial segment $A$ is contained into a proper initial segment of the form $P\setminus \uparrow x$ for some $x\in P\setminus A$. 

As a special consequence of this lemma, note that a J\'onsson poset cannot have a maximal element (if $a$ is a maximal element of $P$ then $P\setminus \{a\}$ is an initial segment of $P$).

 The following proposition gives a description of J\'onsson linear order; the straightforward proof is omitted. 
 
 \begin{proposition} A well founded linear order $L$ is J\'onsson iff and only if its order type is an initial ordinal. Furthermore,  a linear order $L$  of cardinality $\kappa$ is 
 J\'onsson iff $L$ is a lexicographic sum $\sum_{\alpha< \mu} C_{\alpha}$ of chains $C_{\alpha}$ of cardinality less than $\kappa$  indexed by a regular ordinal $\mu$. 
 \end{proposition}

 The relationship between J\'onsson posets and J\'onsson linear orders 
is at the bottom of properties  of  J\'onsson posets:

\begin{proposition}\label{prop:jonsson-extension} Let $P:= (X, \leq)$ be  a poset. Then the following properties are equivalent:
\begin{enumerate}[(i)]
\item  $P$ is J\'onsson; 
\item  Every strengthening of $P$ is J\'onsson; 
\item Every linear extension of $P$ is J\'onsson. 
\end{enumerate}
\end{proposition}

\begin{proof}
$(i)\Rightarrow (ii)$. Let $\leq'$ be a strengthening of the order $\leq$ of $P$. Let $A$ be a proper initial segment of $P':= (X, \leq)$. Then $A$ is an initial segment of $P$ (indeed, if $x\in A$ and $y\leq x$, then,  since $\leq'$ contains  $\leq$, $y\leq'x$ hence $y\in A)$). Hence $\vert A\vert< \vert X\vert$. Thus $P'$ is J\'onsson. 

$(ii) \Rightarrow (iii)$.  Obvious. 

$(iii)\Rightarrow (i)$.  Let $A$ be an initial segment of $P$. Let $B$ be the complement. Let 
$P'$ be the \emph {sum} $A+B$ that is the poset in which the order extends the order $\leq$ on $P$ and every element of $A$ is smaller than  every element of $B$. A linear extension $\leq '$ of this order is a linear extension of $\leq$,  furthermore $A$ is an initial segment of $P':= (X, \leq)$. Since $P'$ is J\'onsson,  $\vert A\vert < \vert X\vert$. Thus $P$ is J\'onsson. \end{proof}

Since any countable J\'onsson linear order has order type $\omega$, this proposition yields the equivalence between $(ii)$ and $(iv)$ of Proposition \ref{label:prop:minimaltype}. 

\begin{lemma}\label{lem:Jonsson-finite} On  a poset of cardinality $\kappa$ the order is J\'onsson whenever  it  is the intersection of less than $\cf(\kappa)$ J\'onsson orders. 
\end{lemma}
 \begin{proof} Let $P:= (X, \leq)$. And let $(\leq_i)_{i<\nu}$, with $\nu<\cf(\kappa)$,  be a family of  $\nu$  orders such that $\leq= \bigcap_{i<\nu} \leq_i$. Let $x\in X$. For $i<\nu$ set $\uparrow_i x:=\{y\in X: x\leq_i y\}$. Then  $\uparrow x= \bigcap_{i<\nu} \uparrow_i x. $   Hence, $X\setminus \uparrow x= \bigcup_{i<\nu} X\setminus \uparrow_i x. $ from which follows:
 $$\vert X\setminus \uparrow x\vert \leq \sum_{i<\nu} \vert X\setminus \uparrow_i x\vert. $$
 Since each $\leq_i$ is J\'onsson,  $\vert X\setminus \uparrow_i x\vert< \vert X\vert$. Since $\nu<\cf (\kappa)$ and $\kappa= \vert X\vert$, \mbox{$\vert X\setminus \uparrow x\vert< \vert X\vert$}. Hence $\leq$ is J\'onsson. 
 \end{proof}

\begin {corollary}\label{cor:Jonsson-dim}
Let $P$ be a poset of cardinality $\kappa$ and let $\mu:= \cf(\kappa)$.  The order on $P$ is the intersection of strictly less that $\mu$ linear J\'onsson orders  if and only if $P$ is the intersection of strictly less than $\mu$  linear orders  and $P$ is J\'onsson.\end {corollary} 

Since any countable J\'onsson linear order has order type $\omega$, this corollary yields Corollary \ref{cor:dim<n}.

We recall that a poset  $P$ is w.q.o. iff all its linear extensions are well ordered \cite{wolk}. Hence, it follows from Proposition \ref{prop:jonsson-extension} that  \emph{a w.q.o.  is J\'onsson iff it is infinite and all its linear extension have the same order type and this order type is an initial ordinal.} If $P$ is w.q.o.  there is a largest ordinal  type  of the linear extensions of $P$, a famous and non trivial result due to de Jongh and Parikh \cite{dejongh-parikh}. This order type, denoted by $\mathbf o(P)$, is the \emph{ordinal length} of $P$.  With this notion,  a w.q.o. poset $P$ is J\'onsson iff $P$ is infinite and $\mathbf o(P)$ is the cardinality of $P$, viewed as an initial ordinal. For an illustration of Lemma \ref{lem:Jonsson-finite}, we mention that \emph{every poset whose order is the intersection of finitely many linear order of ordinal length $\kappa$, where $\kappa$ is an initial ordinal, is a w.q.o. with ordinal length $\kappa$} (see \cite{delhomme-pouzet}).

 \begin{proposition}\label{prop:uniqueness} Every w.q.o.  of infinite cardinality contains an initial segment with the same cardinality which is J\'onsson.
 \end{proposition}

\begin{proof} Let $P$ be a w.q.o. The set of initial segments of $P$, once ordered by inclusion,  is well founded (Higman 1952, see \cite{fraisse}). Among the initial segments of $P$ with the same cardinality, take a minimal one. This is a J\'onsson poset. 
\end{proof}

\subsection{Kearnes-Oman result and cofinality}

Kearnes and Oman \cite{kearnes} proved the following result. 

\begin{theorem}\label{theo:kearnes}
If $P$ is a J\'onsson poset of cardinality $\kappa$ then, for every cardinal $\lambda<\kappa$,  $P$  contains  some principal initial segment of cardinality at least $\lambda$. 
\end{theorem}

This is a significant result in the sense that several apparent strengthenings follow easily from it. We present in Theorem \ref{labtheo:cof} a slight  improvement.

There are  posets with singular cofinality (e.g. an antichain with singular cardinality).  This cannot be the case with a J\'onsson poset. As said in Corollary \ref{cor:cofinality}, J\'onsson posets have a regular cofinality. 

The first reason  of the relevance of the cofinality notion is this: 

\begin{lemma}\label{lem:cofinal-jonsson}For an infinite poset $P$ the following properties are equivalent:
\begin{enumerate} [(i)]
\item $P$ is J\'onsson;
\item Every subset of $P$ of cardinality $\vert P\vert$ is cofinal in $P$; 
\item There is some cofinal subset $C$ of $P$ such that $\vert P\setminus \uparrow x\vert < \vert P\vert$ for every $x\in C$.
\end{enumerate}
\end{lemma}
\begin{proof} $(i)\Rightarrow (ii)$. Let $P$ be a J\'onsson poset and $A$ be a subset of $P$ with cardinality $\vert P\vert $. The cardinality of the initial segment $\downarrow A$ is 
$\vert P\vert $. Since $P$ is J\'onsson, $\downarrow A= P$, hence $A$ is cofinal. $(ii)\Rightarrow (iii)$. Let $C$ be \emph{any}  subset  of  $P$. Let $x\in C$.  Since $ P\setminus \uparrow x$ cannot be cofinal in $P$, $(ii)$ ensures that $ \vert P\setminus \uparrow x\vert <\vert P\vert$. $(iii)\Rightarrow (i)$.  Let  $A$ be a proper initial segment. Let $x\in P\setminus A$. Since $x$  is dominated by some $y\in C$ we have $\vert A\vert \leq \vert P\setminus \uparrow x\vert \leq \vert P\setminus \uparrow y\vert < \vert P\vert$, thus $\vert A\vert < \vert P\vert$. Hence $P$ is J\'onsson. 
\end{proof}

In fact, in every  J\'onsson poset,  some  cofinal subset  is a chain. This is our first result:

\begin{theorem}\label{labtheo:cof}A poset $P$ of infinite cardinality $\kappa$  is J\'onsson iff there is a cofinal chain $C$ with cofinality $\cf(\kappa)$ such that $\vert P\setminus \uparrow x\vert < \vert P\vert$ for every $x\in C$.  \end{theorem}

A straigthforward  proof based on K\"onig's Lemma (Theorem \ref{lem:konig})  is given for level-finite posets  in Subsection \ref{subsection:spectrum}. 
As an immediate corollary of \mbox{Theorem \ref{labtheo:cof}}, we have:
\begin{corollary}\label{cor:cofinality}
 The cofinality  of a   J\'onsson poset  is the cofinality of its cardinality; in particular, this is a regular cardinal. 
\end{corollary}

The proof of Theorem \ref{labtheo:cof} relies on the following two lemma, the first one being well known:
\begin{lemma}\label{lem:basic-cofinality} Every poset $P$ of cofinality $\nu$ in which every subset of cardinality strictly less than $\nu$ has an upper bound has a cofinal well ordered chain of order type $\nu$. 
\end{lemma}
For the second lemma, we introduce the set 
${\NC}(P)$ of non cofinal subsets of a poset $P$. We observe that for a given cardinal $\nu$, every subset of $P$ of cardinality strictly less than $\nu$ has an upper bound in $P$ if and only if ${\NC}(P)$  is a   $(<\nu)$-ideal of subsets of $P$,  that is ${\NC}(P)$  is closed under inclusion and unions of less than $\nu$ members.

\begin{lemma}\label{lemma:cofinality}
Let $P$ be a  J\'onsson poset of cardinality $\kappa$ and let $\nu:= cf(\kappa)$. Then:
\begin{enumerate}
\item  every subset $A$ of cardinality stricly less than $\nu$ has an upper bound;  
\item $\cf(P)= \nu$. 
\end{enumerate}
\end{lemma}

\begin{proof}
\noindent (1) We prove that $\bigcap_{x\in A} \uparrow x$ is non-empty. We have: $$X:= P\setminus \bigcap_{x\in A} \uparrow x= \bigcup_{x\in A} (P\setminus \uparrow x).$$  Hence $\vert X\vert\leq \sum_{x\in A}\vert P\setminus \uparrow x\vert $. Since $P$ is J\'onsson, each member of the sum is strictly less than $\kappa$; since $\vert A\vert <  \nu= cf(\kappa)$, the sum  is stricly less than $\kappa$. Hence $X\not = P$.\\ 
 \noindent $(2-1)$  $\nu\leq \cf(P)$. 
Suppose by contradiction that $\cf(P)< \nu$. Let $(x_{\alpha})_{\alpha< \cf(P)}$ be an enumeration of a cofinal subset of $P$ with size $\cf(P)$. Then $P= \bigcup_{\alpha< \cf(P)}\downarrow\! x_{\alpha}$. Since $P$ is J\'onsson, $\vert x_{\alpha}\vert < \kappa$ for every $\alpha< \cf(P)$. Since $\cf(P) < \cf(\kappa)$ we have 
$\kappa= \vert P\vert \leq \sum_{\alpha <\cf(P)}\vert x_{\alpha}\vert <\kappa$. A contradiction.

\noindent $(2-2)$  $\cf(P)\leq\nu$. For each cardinal $\mu<\kappa$, let $P_{\mu}:= \{x\in P: \vert P \setminus \uparrow x\vert \leq \mu\}$. 

\begin{claim} \label {claim1}There is a set $Z_{\mu}$ of cardinality a most $\nu$ such that $P_{\mu} \subseteq \downarrow Z_{\mu}$. 
\end{claim}

\noindent{\bf Proof of Claim \ref{claim1}}
Let $X\subseteq P_{\mu}$. Then $\vert \bigcup_{x\in X} P\setminus \uparrow x \vert \leq \mu. \vert X\vert$. If $\vert X\vert <\kappa$, $\vert \bigcup_{x\in X} P\setminus \uparrow x\vert < \kappa$ hence $\bigcup_{x\in X} P\setminus \uparrow x \not =P$. Any $z\in P\setminus \bigcup_{x\in X} (P\setminus \uparrow x)= \bigcap_{x\in X} \uparrow x$ dominates $X$. If $\vert P_{\mu}\vert<\kappa$,  set $X:= P_{\mu}$ and  set $Z_{\mu}:= z$. If $\vert P_{\mu}\vert=\kappa$, enumerate it by a sequence $(x_{\nu})_{\nu< \kappa}$. Let
$(\mu_{\alpha})_{\alpha<\nu}$ be a cofinal sequence in $\kappa$. For each $\alpha< \nu$, select $z_{\alpha}$ which dominates $X_{\alpha}:=\{x_{\nu}: \nu< \mu_{\alpha}\}$ and set $Z_{\mu}:= \{z_{\alpha}: \alpha < \nu\}$. 	With that, the proof of Claim \ref{claim1} is complete. 				\hfill $\Box$

We conclude the proof of $(2-2)$ as follows. Since $P$  is J\'onsson, $P = \bigcup_{\alpha< \nu} P_{\mu_{\alpha}}$. Hence, with Claim \ref{claim1}, $P \subseteq \bigcup_{\alpha< \nu}\downarrow  Z_{\mu_{\alpha}} =\downarrow Z$ where $Z:= \bigcup_{\alpha<\nu} Z_{\mu_{\alpha}}$. The set $Z$ is cofinal  and has cardinality at most $\nu$, hence $\cf(P)  \leq \nu$. \end{proof}

\begin{remarks} 
%
\emph{The proof of $(2-2)$ is reminiscent of the proof that if a poset $P$ of cardinality $\kappa$ is well founded,  $\vert \downarrow x \vert < \kappa$ for every $x\in P$ and $\cf(P)= \kappa$  is a  singular cardinal  then  there is some $x\in P$ such that $\cf(P\setminus \uparrow x)= \kappa$ (see \cite{milner-pouzet}). A fact from which it readily follows that \emph{if the cofinality  of a poset is a singular cardinal, the poset contains an infinite antichain}, a result due to the second author and reproduced for example in \cite {fraisse}}. 

\end{remarks}

  \begin{lemma} \label{lem:cof-sum}Let $P$ be a poset  of cardinality $\kappa$.
 \begin{enumerate}
 \item If $P$ has a cofinal chain then 
for every  $\lambda <\kappa$ it contains a principal initial segment of cardinality at least $\lambda$.
\item If $P$ is non-empty with no largest element, then  $P$ contains a subset $Q$ of the same cardinality $\kappa$ as $P$ which is  a lexicographical sum $\sum_{\alpha<\cf(\kappa)} P_{\alpha}$ of  $\cf(\kappa)$ sets $P_{\alpha}$ of cardinality strictly less than $\kappa$ provided that
for every  $\lambda <\kappa$ every proper final segment of $P$ contains a principal initial segment of cardinality at least $\lambda$ and  every subset of $P$ of cardinality strictly less than $\cf(\kappa)$ is majorized. The converse holds if $Q$ is cofinal in $P$. 
\item \label{item-cofinality} If $P$   is J\'onsson then the conditions in Item (2) are satisfied.  
\end{enumerate}
\end{lemma} 
\begin{proof}
$(1)$. Let $C$ be a cofinal chain.  With no loss of generality we may suppose that $C$ is well ordered and that its order type is an initial ordinal $\nu$, with $\nu$ regular. If $\nu= \kappa$ the conclusion holds for $C$ hence for $P$. Suppose $\nu<\kappa$. We have $P= \bigcup_{x\in C}\downarrow x$, hence $\vert P\vert \leq \sum_{x\in C}\vert  \downarrow x\vert $. If the conclusion  does not hold, there is some $\lambda< \kappa$ with $\vert \downarrow x\vert <\lambda$ for every $x\in C$, hence $\kappa\leq \lambda\cdot \nu< \kappa$, which is impossible.

$(2)$. Suppose that $P$ contains a lexicographical sum  $Q$ as described in the sentence. Let $x\in P$. Then  $\uparrow x$ contains the sum $\sum_{\alpha_0<\alpha<\cf(\kappa)} P_{\alpha}$ where $\alpha_0$ is such that $\uparrow x\cap P_{\alpha_0}\not = \emptyset$. Hence, $\vert \uparrow x \vert \geq \kappa$. If $Q$ is cofinal and  $A$ is any subset of cardinality strictly less than $\cf(\kappa)$, then $A$ is majorized. On an other hand, suppose that the two conditions are satisfied. Let $\nu:= \cf(\kappa)$.  If $\nu= \kappa$, then  according to Lemma \ref{lem:basic-cofinality}, $P$ contains a cofinal chain with order type $\kappa$ and $P$ contains a poset $Q$ as described. Suppose that $\nu<\kappa$. Let $(\kappa_{\alpha})_{\alpha< \nu}$ be an increasing sequence of cardinal numbers whose  supremum is $\kappa$. We define a sequence $(x_{\alpha})_{{\alpha}<\nu}$ of elements of $P$ such that $\vert Q_{\alpha}\vert \geq \kappa_{\alpha}$ where  \mbox{$Q_{\alpha}:=  \downarrow x_{\alpha} \bigcap (\bigcap_{\beta<\alpha} \uparrow x_{\beta})$}. 

$(3)$.   If $P$ is J\'onsson then every non-empty final segment of $P$ is J\'onsson too. Hence,  according to Theorem \ref{theo:kearnes},   the first condition of item $(1)$ holds. Now, according to item $(1)$ of Lemma \ref{lemma:cofinality}, every subset of size strictly less than $\cf(\kappa)$ is majorized.  From this,  $P$ contains a poset $Q$ as described. Note that since $Q$ has the same cardinality as $P$ and $P$ is J\'onsson, $Q$ is cofinal in $P$. \end{proof}

\begin{remark}With the help of  Item $(1)$ of Lemma \ref {lem:cof-sum}, Theorem  \ref{theo:kearnes} follows immediately from  Theorem \ref{labtheo:cof}.  On an other hand, if $P$ is J\'onsson then, according to Item (\ref{item-cofinality} of  Lemma  \ref{lem:cof-sum} above,   it contains a poset $Q$ as described in Item $(2)$. According to Lemma \ref{lem:cofinal-jonsson}, this poset $Q$ is cofinal in $P$, hence $P$ has a cofinal  chain thus  Theorem \ref{labtheo:cof} holds.  
\end{remark}

An improvement  of Item (\ref{item-cofinality}) is given in Theorem \ref {theo:maintheo}. 

%
%

\subsection{Purity}
In order  to describe (some)  J\'onsson posets we start with  the following notion.

 \begin{definition}\label{def:pure} A  poset $P$ is \emph{pure} if every proper initial segment $I$ of $P$ is strictly  bounded above (that is some $x\in P\setminus I$ dominates $I$). 
 \end{definition}
 
 This condition amounts to the fact that every non cofinal subset of $P$ is strictly bounded above (indeed, if a subset $A$ of $P$ is not cofinal, then $\downarrow A\not = P$ hence from purity, $\downarrow A$, and thus $A$, is strictly bounded above. The converse is immediate). 

An equivalent condition is this:

$\bullet $ \emph{For every $x\in P$ there is some $y\geq x$ such that $P\setminus \uparrow x\subseteq \downarrow y$}. 

Every poset with a largest element is pure. Every chain is pure.  \emph{Every  pure poset has a cofinal chain}. This last fact is consequence of the following lemma.  

\begin{lemma} \label{puritysequence} Let $P$ be a poset with infinite cofinality $\nu$. Then $P$ is pure iff it contains an  increasing cofinal sequence $(x_{\alpha})_{\alpha<\nu}$   such that 
\begin{equation}\label{eq:purity}
P\setminus \uparrow x_{\alpha}\subseteq \downarrow x_{\alpha+1}
\end{equation}
for all ordinal $\alpha$ such that $\alpha< \nu$. 
\end{lemma}
\begin{proof} Suppose that $P$ contains such a sequence. Let $x$ any element of $P$. Then $x\leq x_{\alpha}$ for some $\alpha <\nu$. Let $y:= x_{\alpha+1}$. We have $P\setminus \uparrow x\subseteq P\setminus \uparrow x_{\alpha} \subseteq \downarrow x_{\alpha+1}= \downarrow y$.  This proves that $P$ is pure. Conversely, suppose  that $P$ is pure.  As any poset, $P$ contains a cofinal sequence $(y_{\alpha})_{\alpha< \nu}$ which is non decreasing in the sense that $y_{\beta} \not \leq y_{\alpha}$ for $\alpha<\beta$. From the purity of $P$,  we may extract a subsequence satisfying Inequality (\ref{eq:purity}).  Indeed, define $\varphi: \nu\rightarrow \nu$ as follows. Set $\varphi(0):=0$. Suppose $\varphi$ be defined for all $\beta'$ with $\beta'<\beta$. The set  $\{y_{\beta'}: \beta'<\beta\}$ cannot be cofinal in $P$,  hence there is some $y_{\gamma} \in P\setminus \downarrow \{y_{\beta'}: \beta'<\beta\}$. If $\beta$ is a limit ordinal distinct of $0$, set $\varphi(\beta)= \gamma$. If not, $\beta= \beta'+1$ and  there is some $y\geq y_{\beta'}$ such that $P\setminus \uparrow y_{\beta'}\subseteq \downarrow y$. If there is no element strictly above $y$ then $y$ is the largest element of $P$,  a case we have excluded. Hence  $y< y_{\delta}$ for some $\delta$ and in fact $\delta >\beta'$. We set $\varphi(\beta)= \delta$. Setting $x_{\alpha}:= y_{\varphi(\alpha)}$,  Inequality (\ref{eq:purity}) is then satisfied. Let us check that the sequence is increasing. If not, let $\alpha < \beta$ with $x_{\alpha} \not \leq x_{\beta}$. We have $x_\beta\in P\setminus \uparrow x_{\alpha} \subseteq \downarrow x_{\alpha+1}$ hence $x_{\beta} \leq x_{\alpha+1}$ contradicting the fact that the sequence $(y_{\alpha})_{\alpha<\nu}$ is non decreasing. 
\end{proof}

$\bullet$ $(a)$ \label{bullet:strenghtening-pure} \emph{If a poset  $P$ is a strengthening of a pure poset $Q$ then $P$ is pure. The same conclusion holds if $P$ is a cofinal subset of a  pure poset}.
 
Indeed, let  $A$ be a proper initial segment of $P$. Then $A$ is an initial segment of $Q$ hence it is proper. Since $Q$ is pure, $A$ is majorized by some element $x$. Since $P$ is a strengthening of $Q$, $x$ majorizes $A$ in $P$. Thus $P$ is pure. Now if $P$ is a cofinal subset of $Q$ and $A$ a proper initial segment of $P$ then $\downarrow A\not = Q$ hence $\downarrow A$ is bounded above by some $x\in Q$. Such an element belongs to $P$, hence  is a bound of $A$ in $P$.

Pure posets are not necessarily J\'onsson. However, pure posets and J\'onsson posets are not far apart:

$\bullet$ $(b)$ \emph{If a poset with no largest element is pure it contains a cofinal subset which is J\'onsson} (indeed, it contains a cofinal chain).

$\bullet$ $(c)$ \emph{If $P$ is pure then $P$ is J\'onsson iff   $P$ is infinite and $\vert \downarrow x\vert <\vert P\vert$ for every $x\in P$}. 

Indeed, if $P$ is J\'onsson then trivially, $\vert \downarrow x\vert <\vert P\vert$ for every $x\in P$. Conversely, let $A$ be a proper initial segment of $P$. Since $P$ is pure, $A$ is bounded above, that is $A\subseteq \downarrow x$ for some $x\in P$. According to the second condition, $\vert \downarrow x\vert< \vert P\vert$ hence $\vert  A\vert <\vert P\vert$. This proves that $P$ is J\'onsson.

Purity and the condition above on principal initial segments imply J\'onsson. 
The converse holds if the cardinality is regular. In particular, every minimal poset is pure. This is  a consequence of Theorem \ref{labtheo:cof} or $(3)$ of Lemma  \ref{lem:cof-sum}.

\begin{theorem}\label{lem:pure-regular} If the cardinality of  $P$ is an infinite regular cardinal $\kappa$, then $P$ is J\'onsson  if and only if $P$ is pure and $\vert \downarrow x\vert <\kappa$ for every $x\in P$. \end{theorem}
\begin{proof}
Suppose  that $P$ is J\'onsson. Trivially, $\vert \downarrow x\vert <\kappa$ for every $x\in P$. Now, we show that $P$ is pure. Let $Q$ be a proper initial segment of $P$ and  let $\lambda:= \vert Q\vert$. According to Lemma \ref{lemma:cofinality}, $P$ contains a well ordered cofinal chain $C$ of order type $\kappa$. Since $C$ is cofinal in $P$, every $x\in Q$ is majorized by some $y_x\in C$. Let $Q':=\{y_x: x\in Q\}$. We have $\vert Q'\vert \leq \vert Q\vert = \lambda$. Since $P$ is J\'onsson, $\lambda< \kappa$. Since $\kappa$ is regular, $Q'$ is not cofinal in $C$, thus it is majorized and hence $Q$ is majorized. Thus  $P$ is pure. \end{proof}

%
%
%

If $P$ is not pure,  it could happen that by deleting some initial segment the remaining set is pure. But this is not general. 

$\bullet$ $(d)$ \label{lem:initial} \emph{Let $P$ be  a poset and $Q$ be a proper  initial segment of $P$. Then $P$ is J\'onsson iff $P\setminus Q$ is J\'onsson, $\vert Q\vert <\vert P$ and $Q\subseteq \downarrow (P\setminus Q)$}.

%
%

%

\begin{theorem}\label{theo:maintheo}
Let $P$ be a  poset with  infinite  cofinality $\nu$. Then $P$ is pure iff $P$ is a strengthening of a lexicographical  sum $\sum_{a\in K }P_{a}$ where $K$  is  a chain of order type $\nu$ if $\nu$ if uncountable or a minimal poset if $\nu$ is countable and every  $P_{a}$ is a non empty poset. If $P$ is pure, then $P$ is J\'onsson iff each member $P_{a}$ of the sum above has cardinality $\kappa_{\alpha}$ strictly less than $\vert P\vert$. 
\end{theorem}

\begin{proof} 
A lexicographical sum as above is pure; thus from $\bullet$ $(a)$, every strengthening is pure. Moreover,  if each member of the sum has cardinality less that $\vert P\vert$, the sum is J\'onsson  hence, by Proposition 
\ref{prop:jonsson-extension},  every extension is J\'onsson. 

For the converse, let $N:= (\nu, \leq_2)$ where $\leq_2$ is the order defined on $\nu$ by $\alpha\leq_2\beta$ if $\alpha=\beta$ or $\alpha+2\leq \beta$. Then $N$ is pure and J\'onsson. 

\begin{claim}\label{claim:sum} If $P$ is pure then $P$ is a strengthening of a   lexicographical  sum $\sum_{a\in N }P_{a}$ where each $P_{a}$ is  a non empty poset 
\end{claim}

\noindent{Proof of Claim \ref{claim:sum}.} Let  $(x_{\alpha})_{\alpha<\nu}$ be the sequence given by Lemma \ref{puritysequence}.  For every $x\in P$,  let $h(x)$ be the least ordinal $\alpha$ such that $x \leq x_{\alpha}$, and for $\alpha<\nu$, let $P_{\alpha}:= \{x\in P: h(x)= \alpha\}$.  The order on $P$ extends the order on the sum $\sum_{a\in N }P_{a}$ provided that for every $x\in P_{\alpha}$, $y\in P_{\beta}$,  $\alpha+2 \leq \alpha$ implies $x\leq y$ in $P$. This is trivial:  we have $y\geq x_{\alpha}$, (otherwise $y\not \geq   x_{\alpha}$ and  since $P\setminus \uparrow x_{\alpha}\subseteq \downarrow x_{\alpha+1}$, we have $y \leq x_{\alpha +1}$ giving $h(y)\leq \alpha+1$ while $h(y)=\beta >\alpha+1$). Since $x\leq x_{\alpha}$  we get $x\leq y$ by transitivity,   as claimed.\hfill $\Box$

If $\nu=\omega$ then  $N$ is minimal  and   the sentence   in the theorem  holds. If $\nu$ is uncountable,  let $(\mu_{\lambda})_{\lambda<\nu}$ be an increasing  cofinal sequence in $\nu$ of limit ordinals; set $N_{\lambda}:=\{\alpha : \mu_{\lambda} \leq \alpha <\mu_{\lambda+1}\}$. Then $N$ is the lexicographical sum of its restrictions to the $N_{\lambda}$'s indexed by the chain $\nu$. Setting $Q_{\lambda}:=\bigcup_{a\in N_{\lambda} }P_{a}$, we get that $P$ is the lexicographical sum of the $Q_{\lambda}$'s.  
\end{proof}
Note that if $P$ is a pure J\'onsson poset of cardinality $\kappa$ with $cf(\kappa)> \aleph_0$, one can  easily show  that the incomparability graph of $P$ decomposes into at least $cf(\kappa)$ connected components,  each of cardinality strictly less than $\kappa$. This yields an other proof of Theorem \ref{theo:maintheo} in this  case. 

With Theorem \ref{lem:pure-regular} and Theorem \ref{theo:maintheo},  we have:

\begin{corollary}\label{cor:width}
If $P$ is a J\'onsson poset and $\kappa:=\vert P\vert $ is a successor cardinal then $P$ is the union of strictly less than $\kappa$ chains. 
\end{corollary}

\begin{remark}\emph{It is not  true that an uncountable  pure and J\'onsson poset $P$ with  countable cofinality is the lexicographical sum $\sum_{n\in M }P_{n}$ where $M$ is  a minimal poset,  each $P_{n}$ is a non empty poset of cardinality $\kappa_{n}$ stricly less than $\kappa:= \vert P\vert$ and the supremum of $\kappa_{n}$ is $\kappa$. The reason is that a strengthening of such a poset is pure and J\'onsson but not necessarily a lexicographical sum. We give an example below}. 
\end{remark}
Let $Q$ be the lexicographical sum $\sum_{n\in M }Q_{n}$ where $M$ is  the poset on the set $\N$ of non-negative integers, with $n\leq m $ if either $n=m$ or $n+2 \leq m$,  each $Q_{n}$ is an antichain of  cardinality the $n$-beth number $\beth_n$ (where $\beth_0:= \aleph_0$, $\beth_{n+1}:= 2^{\beth_{n}}$). Trivially, $Q$ is pure and J\'onsson. We define a strengthening $P$ of $Q$ by adding just some well choosen comparabilities between pairs of consecutive levels $Q_n$ and $Q_{n+1}$, for $n\in \N$. Hence $P$ will be pure and J\'onsson. In order to do so, we suppose that $Q_{n+1} = \powerset (Q_{n})$ and we add all pairs $(x,y)$ such that $x\in Q_n$, $y\in Q_{n+1}$ and $x\in y$. We claim that the resulting poset $P$ does not decompose into a non trivial sum. Indeed, otherwise some factor  of the sum would be a \emph{proper autonomous subset} of $P$, that is,  a subset $F$ distinct from the empty set, any singleton and the whole set, and such that for every $x, x'\in F$ and $y\not \in F$, $x\leq y$ iff $x'\leq y$ and also $y\leq x$ iff $y\leq x'$.  This is impossible. For each non-negative integer, the ordering induces a bipartite graph  on $Q_{n} \cup Q_{n+1}$ which  is 
connected and such that distinct vertices have distinct neighborhoods. Since this graph has more than three vertices, it has no proper autonomous subset (e.g. see Proposition 1 of \cite{pouzet-zaguia}). Hence, for each $n$, $F\cap (Q_{n} \cup Q_{n+1})$ is either empty, a singleton or $Q_{n} \cup Q_{n+1}$. It is easy to see that this   conclusion extends  to $P$. This proves our claim.

\subsection{Semiorders and J\'onsson posets}\label{subsection semiorders}
The poset $N:= (\nu, \leq_2)$ which appears in the proof of Theorem \ref{theo:maintheo} is an example of semiorder. We examine below the role of these orders in the present study.  

Posets which do not embed the direct sum $2 \oplus 2$  of two $2$-element chains are called \emph{interval orders}, whereas posets which do not embed  $2 \oplus 2$   nor $3 \oplus 1 $, the direct sum of a $3$-element chain and a $1$-element chain, are called \emph{semiorders}.  Semiorders were introduced and applied in mathematical psychology by Luce \cite{luce}. For a wealth of information about interval orders and semiorders the reader is referred to \cite{fi-book} and \cite{pirlotvincke}.
%
%

The name interval order comes from the fact that the order of a poset $P$ is an interval order   iff $P$ is isomorphic to a subset $\mathcal J$ of the set $Int(C)$ of
nonempty intervals of a totally ordered set $C$, ordered as follows: if $I, J\in
Int(C)$, then
\begin{equation}
I<J \mbox{  if  } x<y  \mbox{  for every  } x\in I  \mbox{  and every  } y\in J.
\end{equation}

This result is due to Fishburn \cite{fi}. See also Wiener \cite{wiener}.

The Scott and Suppes Theorem \cite{scottsuppes} states that the order of a \emph{finite} poset $P$ is a semiorder if and only if $P$ is isomorphic to a collection if intervals of length 1 of the real line, ordered as above. Extension of this result to infinite semiorders have been considered (see \cite{candeal}).

Interval orders and semiorders can be characterized  in terms of the  quasi-orders $\leq_{pred}$ and $\leq_{succ}$ associated with a given order. They are defined as follows. Let $P$ be a poset. Set $x\leq_{pred}y$ if $z<x$ implies $z<y$ for all $z\in P$ and set $x\leq_{succ}y$ if $y<z$ implies $x<z$ for all $z\in P$. The  relations $\leq_{pred}$ and $\leq_{succ}$ are  quasi-orders.  The strict orders associated to $\leq_{pred}$ and $\leq_{succ}$ extend  the strict order associated to  $\leq$, that is:
\begin{equation}\label{equa:1}
x< y\Rightarrow x<_{pred}y  \; \text{and} \; x<_{succ}y 
\end{equation}
for all $x,y \in P$.


We recall the following result (see Theorems 2 and  7 of \cite{rabinovitch}).
\begin{lemma} \label{intervalorder} Let $P$ be a poset. Then $P$ is an interval order if and only if the quasi-order  $\leq_{pred}$ is a total quasi-order; equivalently $\leq_{succ}$ is a total quasi-order. Furthermore, $P$ is a semiorder if and only if the quasi-order intersection of $\leq_{pred}$ and $\leq_{succ}$ is total. 
\end{lemma}

Note that the intersection of $\leq_{pred}$ and  $\leq_{succ}$ can be total and these quasi-orders can be distinct For an example, look at the direct sum of a $2$-element chain and a $1$-element chain). For an other example, if $\leq$ is the order $\leq_2$ on the ordinal $\nu$, then $\leq_{\succ}$ coincide with the natural order on $\nu$, $\leq_{pred}$ coincide with the natural order on all pairs distinct from the pair $(0,1)$ but do not distinguish between $0$ and $1$. Orders  $\leq$ such that  $\leq_{pred}$ and $\leq_{succ}$ are total orders and equal are studied in  \cite{pouzet-zaguia2}, under the name of  \emph{threshold orders}.

We present an other characterization which is relevant to our purpose. 

\begin{theorem}\label{thm: intervalorder} A poset $P:= (X, \leq)$ is an interval order, resp. a semiorder,  iff there is an order-preserving map $h$ from $P$ into a chain $K$ and a map, resp. an order-reversing  map,  $\Psi: K\rightarrow \mathcal F(K)$  such that:
\begin{enumerate}
\item $k\not \in \Psi(k)$ for every $k\in K$; 
\item   $x<y$ in $P$  iff  $h(y) \in \Psi (h(x))$ for every $x,y \in P$.

\end{enumerate}
\end{theorem}
\begin{proof} 

Suppose that $P$ is an interval order. By Lemma \ref{intervalorder} we may select a  total order $\preceq$  included into  $\leq_{pred}$. Set $K:= (X, \preceq)$, $h$ be the identity and  for each $x\in X$,  set $\Psi(x):= \{y\in X: x<y\}$. 
Then  $\Psi(x)$ is a final segment of $(X, \preceq)$;  indeed, first  $\Psi(x)$ is a final segment of $(X, \leq_{pred})$ (if  $y\in \Psi(x)$ and $z$ are such that $y\leq_{pred}z$, then  since $x<y$ we have $x<z$ hence $z\in \Psi(x)$, proving our assertion). Next, since $\preceq$ is included into $\leq_{pred}$, it follows that $\Psi(x)$ is a final segment of $K:= (X, \preceq)$. Conditions $(1)$ and $(2)$ hold trivially. 
If  $P$ is a semiorder then,  according to Lemma \ref{intervalorder},  the quasi-order intersection of $\leq_{pred}$ and $\leq_{succ}$ is total. Hence, we may select a total order $\preceq$  included   into this intersection. In this case, the map  $\Psi: P \rightarrow K:= (X, \preceq)$ is order decreasing (indeed, let $x\prec y$ and let $z\in \Psi(y)$; by definition, $y<z$. Since $x\prec y$, $x<z$ hence $z\in \Psi (x)$ proving 
$\Psi(y)\subseteq \Psi(x)$). 

 In order to prove that the converse holds, suppose  that there is an order-preserving map  $h$ from $P$ into a chain $K$ and  a map $\Psi: K\rightarrow \mathcal F(K)$  such that:
$(1')$: $k\not \in \Psi(k)$ for every $k\in K$ and  
$(2')$:  $x<y$ in $P$ whenever  $h(y) \in \Psi (h(x))$ for every $x,y \in P$.

Set  $x<_{\Psi} y$ if $h(y)\in \Psi (h(x))$. Hence Condition $(2')$ ensures that $x<_{\Psi} y$ implies $x<y$. 

With the claims below, we prove that the relation $<_{\Psi}$ is a strict order and the corresponding order $\leq_{\Psi}$ an interval order. Since Condition $(2)$ expresses that  $<_{\Psi}$ and $<$ coincides, this proves that the converse of the lemma holds.

\begin{claim}\label{claim:strictorder}
The relation $<_{\Psi}$  is a strict-order on $X$.
\end{claim}

\noindent {\bf Proof of Claim \ref{claim:strictorder}.} First, this relation is irreflexive: from $(1')$, \mbox{$h(x)\not \in \Psi(h(x))$}, that is $x\not <_{\Psi} x$. Next, it is transitive. Indeed, suppose  $x<_{\Psi} y$ and $y<_{\Psi }z$. We have  $h(y) \in \Psi (h(x))$ and  $h(z) \in \Psi (h(y))$. Since $K$ is a chain, $\Psi (h(x))$ and  $ \Psi (h(y))$ are comparable w.r.t. set inclusion. If $\Psi (h(x))\subseteq \Psi (h(y))$, then since $h(y) \in \Psi (h(x))$ we have $h(y) \in \Psi (h(y))$,  a fact which is exluded by $(1')$. Hence $ \Psi (h(y))\subset \Psi (h(x))$. Since $h(z) \in \Psi(h(y)$, we have $h(z)\in \Psi(h(x))$ that is $x<_{\Psi} z$, proving the transitivity.  \hfill              	$\Box$

 Let $\leq_{\Psi}$ be the order associated to $<_{\Psi}$, that is $x\leq_{\Psi}y$ if $x=y$ or 
 $x<_{\Psi}y$; let $\leq_{{\Psi}{pred}}$ and $\leq_{{\Psi} {succ}}$ be the preorder "pred" and "succ" associated with $\leq_{\Psi}$. Set $\Phi(x):= \{y\in X: h(y) \in \Psi (h(x))\}$.

 \begin{claim}\label{claim:pred} If $h(x)\leq h(y)$ then $x\leq_{{\Psi}{pred}}y$. 
\end{claim}
\noindent {\bf Proof of   Claim \ref{claim:pred}.} Let $z\in X$ such that $z<_{\psi} x$. We need to prove that $z<_{\psi} y$ that is $h(y)\in \psi (h(z))$. Since $z<_{\psi} x$, we have $h(x) \in \psi (h(z))$. Since $\psi (h(z))\in \mathcal F(K)$ and $h(x)\leq h(y)$ we have $h(y)\in \psi (h(z))$, as required. 
\hfill							$\Box$
\begin{claim} \label{claim:succ}  $\Psi (h(y))\cap h(P) \subseteq \Psi (h(x))$ iff $x\leq_{{\Psi}{succ}}y$.\end{claim}
\noindent {\bf Proof of Claim \ref{claim:succ}.}
Let $z\in X$ such that $y<_{\Psi} z$. This means $h(z)\in \Psi (h(y))$. Since $\Psi (h(y))\cap h(P) \subseteq \Psi (h(x))$, it follows $h(z)\in \Psi(h(x))$ that is $x<_{\Psi(h(x))} z$. Hence $x\leq_{{\Psi}{succ}}y$ as claimed. Conversely, suppose $x\leq_{{\Psi}{succ}}y$. Let $z\in X$ such that  $h(z)\in \Psi (h(y))\cap h(P)$. We have $y<_{\Psi} z$. Since $x\leq_{{\Psi}{succ}}y$ it follows $x<_{\Psi} z$, that is $z\in \Psi (h(x))$, proving that $\Psi (h(y))\cap h(P) \subseteq \Psi (h(x))$. 
\hfill														$\Box$
\begin{claim}  \label{claim:intorder} The order $\leq_{\Psi}$ is an interval order. This is a semiorder provided that  $h$ is order-reversing.  
\end{claim}

\noindent {\bf Proof of Claim \ref{claim:intorder}.}
Since $K$ is a chain, the images via $h$ of any two elements $x$ and $y$ of $P$ are comparable in $K$. According to Claim \ref{claim:pred},  $x$ and $y$ are comparable via the quasi order 
$\leq_{{\Psi}{pred}}$, hence from Lemma \ref{intervalorder}, $\leq_{\Psi}$ is an interval order.  \hfill 						$\Box$

With this claim the proof is complete. \end{proof}


\begin{lemma}\label{lem:semiorder}Let $Q$ be  semiorder. 
\begin{enumerate}
\item If $Q$ has  no maximal element then it is pure.
\item  $Q$ is J\'onsson  iff $Q$ has no maximal element and the order on $Q$  has a strengthening into a J\'onsson linear order.
\end{enumerate}
 \end{lemma}
\begin{proof}
$(1)$. If $Q$ is empty, it is pure. Suppose that $Q$ is non-empty. Our aim is to prove that for every $x\in Q$ there is some $y\in Q$ which majorizes $Q\setminus \uparrow x$. Let $x\in Q$. Since $Q$ has no maximal element, $x$ is not maximal hence there is some $x'\in Q$ with $x<x'$; again $x'$ is not maximal, hence there is some $y\in Q$  with $x'<y$. This element 
$y$  will do. Indeed, let $z\in Q\setminus \uparrow x$. If $z\not \leq y$ then $z$ is incomparable to $x,x'$ and $y$, hence $Q$ contains a $3 \oplus 1 $, contradicting the fact that $Q$ is a semiorder. 

$(2)$ A J\'onsson poset has no maximal element (cf. Lemma  \ref{lem:observation}) and every strenghtening  is  J\'onsson (Proposition \ref{prop:jonsson-extension}).   For the converse, our aim is to prove that $\vert Q\setminus \uparrow x\vert <\vert Q\vert $ for every $x\in Q$. Let $x\in Q$.  Since $Q$ is a semiorder with no maximal element then, according to $(1)$,  it is pure, hence  there is some element $y\in Q$ such that $Q\setminus \uparrow x\subseteq \downarrow y$. If $L$ is any strengthening of the order on $Q$ we have $\downarrow y\subseteq \downarrow_L y$. If $L$ is a J\'onsson order, we have $\vert \downarrow_L y\vert <\vert Q\vert$. Hence, $\vert Q\setminus \uparrow x\vert< \vert Q\vert$. Proving that $Q$ is J\'onsson.  
%
%
\end{proof}

\begin{theorem}\label{thm:semiorder} A poset $P$ is pure and J\'onsson  iff the order on $P$ is a strengthening  of a  semiorder with no maximal element and has a strengthening into a J\'onsson linear order. \end{theorem}
\begin{proof}
Suppose that the order on $P$ is a strengthening  of a  semiorder $Q$ with no maximal element and has a strengthening into a J\'onsson linear order. According to  Lemma \ref {lem:semiorder},  $Q$ is pure and J\'onsson. According to  $\bullet$$(a)$ and Proposition \ref{prop:jonsson-extension}, $P$ is pure and J\'onsson. Conversely, suppose that  the order on $P$ is pure and  J\'onsson.  Apply  Claim \ref{claim:sum} of Theorem \ref{theo:maintheo}, $P$  is a strengthening of a lexicographical  sum $\sum_{a\in N}P_{a}$ where $N:= (\nu, \leq_2)$. We may suppose that the $P_{a}$'s  are antichains; since  $N$  is a semiorder with no largest element, this lexicographical sum is a semiorder with  no maximal element. Since $N$ has a  linear extension of order type $\nu$, $P$ has a strengthening into a J\'onsson order. 
  \end{proof}

\subsection{Conclusion}
Since every  countable J\'onsson poset is  pure and every countable J\'onsson chain has order type $\omega$, it follows from Theorem \ref {thm:semiorder} above that  a countable poset $P$  is J\'onsson   iff the order of $P$ is between   a semiorder with no maximal element and a linear order of type $\omega$. This is Proposition \ref{prop:semiorder}. 

%
%

If the cardinality $\kappa$ of $P$ is uncountable, there are two cases to consider: $\kappa$ is regular, $\kappa$  is singular. 

If $\kappa $ is  a regular cardinal, then by Theorem \ref{thm:semiorder},  Proposition \ref{prop:semiorder} extends  verbatim:
 $P$ is J\'onsson  iff the order of $P$ is between   a semiorder with no maximal element and a J\'onsson linear order.
 
 But,  in this case, $\kappa$ being uncountable, we have a much more precise result:  $P$ is a lexicographical sum indexed by the ordinal  $\kappa$ of posets of cardinality less than $\kappa$ (cf. Theorem \ref{theo:maintheo}). 
 
 If $\kappa$ is singular,  then 
 
  $\bullet$ either   $\cf(\kappa)$ is uncountable and this conclusion still holds provided that $P$ is pure, 
 
 $\bullet$ or   $\cf (\kappa)= \aleph_0$. In this case, $P$ is a pure J\'onsson poset iff the order of $P$ is a strengthening of a lexicographic sum $\sum_{a\in K} P_a$ where $K$ is a minimal poset and $\vert P_a\vert <\kappa$ for every $a\in K$ (Theorem \ref{theo:maintheo}).

\begin{problem} Find a  characterization of   J\'onsson posets of singular cardinality $\kappa$. \end{problem}

%


\section{Uniformity}

Behind the properties of  a pure and J\'onsson poset $P$ (as in $(v)$ of Proposition \ref{label:prop:minimaltype} or in the proof of Claim \ref{claim:sum} of Theorem \ref{theo:maintheo} ) are the properties of  a function $h$ from $P$ into the ordinal numbers. This suggests the following  development. 

\subsection{Uniformity and purity}
\begin{definitions}\label{def:uniform} 
Let $P$ be a poset and $h$ be an order-preserving map from $P$ onto a chain $K$.  We say that $P$ is:

$\bullet$  \emph{${h}$-weakly uniform} if there is a map $\varphi: K\rightarrow K$ such that: 

\begin{equation}\label{weakly uniform}
 \forall x,y\in P (  \varphi (h(x))< h(y) \Longrightarrow x< y). 
\end{equation}

$\bullet$  \emph{$h$-uniform} if  there is a map $\varphi: K\rightarrow K$ such that: 
\begin{equation}\label{uniform}
\forall x,y\in P (\alpha\in K,  h(x)\leq \alpha \; \text{and}\;  h(y)> \varphi (\alpha)  \Longrightarrow x< y).
\end{equation}

$\bullet$ \emph{${h}$-minimal} if  the image of every non cofinal subset of $P$ is non cofinal in $K$, that is 
\begin{equation}\label{minimal}
\downarrow \{h(x): x\in A\} \not = K
\end{equation} for every proper initial segment $A$ of $P$. 
\end{definitions}

These notions, with a slight variation,  were originally defined in \cite {assous-pouzet-Nov86} for well founded posets with the height function ${\bf h}_P$ as an order preserving function. 

For an example, if  $P$  is  $h$-uniform with $\varphi$ being the identity then it is the lexicographical sum of posets indexed by the chain $K$.  
The existence of a map  $\varphi$ satisfying Condition (\ref{weakly uniform}) is a specialization of part of the conditions in Theorem \ref{thm: intervalorder}. Indeed, to $\varphi$  associate the map $\psi: K\rightarrow \mathcal F(K)$, defined by setting 
$\psi (k):= ]\varphi (k) \rightarrow [$. Then, by construction $k\not \in \varphi(k)$ for all $k\in K$. And  Condition (\ref{weakly uniform}) amounts to $x<y$ whenever $h(y)\in \psi (h(x))$. 
%
%
%
%
%
%
A map $\varphi$ witnessing that $P$ is $h$-weakly uniform or $h$-uniform is \emph{extensive}, that is $\varphi(\alpha)\geq \alpha $ for every $\alpha\in K$ (if, otherwise $\varphi(\alpha)<\alpha$ for some $\alpha$ then let  $x$ and $y$ such that $h(x)= \alpha$ and $h(y)= \varphi(\alpha)$;  we have $\varphi(h(x)) = {h(y)}$. Since $h$ is order preserving, $x\not < y$, so neither Condition (\ref{weakly uniform}) nor Condition (\ref{uniform}) may hold).    If $P$ is $h$-uniform and $K$ well-ordered, we may suppose that some  map  $\varphi$ witnessing it is order preserving: indeed, for each $\alpha\in K$, set $P_{\alpha}:= h^{-1}(\alpha):= \{x\in P: h(x)=\alpha\}$, $P_{\leq \alpha}:=\bigcup_{\gamma\leq\alpha} P_{\gamma}$,  $P_{\geq \alpha}:=\bigcup_{\gamma\geq\alpha} P_{\gamma}$ and let $\beta$ be the  least member of $K$ such that 
$P_{\leq \alpha}$  is dominated by every element of $P_{\geq\beta}$. This process
defines  an order preserving map for which $P$ is $h$-uniform. Clearly, if $P$ is $h$-uniform then it is $h$-weakly uniform. Conversely, if $P$ is $h$-weakly uniform and some $\varphi$ witnessing it is order preserving, then $P$ is $h$-uniform. Also, if $\varphi$ witnesses that $P$ is $h$-weakly uniform, and $K$ is well ordered,  set $\varphi^*(\alpha):= \Sup_{\nu< \alpha}\varphi(\nu)$;  if $\varphi^*(\alpha)\in K$ for every $\alpha\in K$ then $P$ is $h$-uniform. In particular,  if the order type of $K$ is a regular initial ordinal  then $P$ is $h$-uniform.


\begin{theorem}
A poset $P$ with no largest element  is pure if and only if it is  $h$-uniform for some order-preserving map $h$ from  $P$ into some infinite limit ordinal $\nu$. If $P$ is well founded, the  height function will do. 
\end{theorem}
\begin{proof}

Suppose that $P$ is  $h$-uniform. We prove that it is pure. Let $A$ be a non cofinal subset of $P$. Pick  some $x\in P\setminus \downarrow A$. Let $\alpha:= h (x)$. Since $P$ is $h$-uniform, $\downarrow A$ is disjoint from $P_{>\delta}:=\bigcup_{\beta > \delta} P_{\beta}$ where  $\delta:=\varphi(\alpha)$, $\varphi$ witnesses that $P$ is $h$-uniform, $P_{\beta}:= h^{-1}(\beta)$.  Hence,  $\downarrow A \subseteq  P_{\leq \delta}:= \bigcup_{\gamma\leq\delta} P_{\gamma}$. Since $P$ is $h$-uniform, any $y\in  P_{>\delta}$ dominates $\subseteq  P_{\leq \delta}$, hence dominates $A$. Hence $P$ is pure.  For the converse, let $\nu:=\cf(P)$. Since $P$ is pure with  no largest element, $\nu$ is an infinite limit ordinal. Let $h:P \rightarrow \nu$ be defined as in the proof of   Claim \ref {claim:sum}.  Let $x, y\in P$ with $h(x) \leq \alpha$ and 
$h(y) > \alpha+1$. Then, we have $x<y$. Hence,  $P$ is $h$-uniform with $\varphi:  \nu\rightarrow \nu$ defined by $\varphi(\alpha):= \alpha+1$.  

Suppose that $P$ is well founded and  is pure. To prove that $P$ is $\mathbf {h}_P$-uniform, it suffices to prove that for every $\alpha<\mathbf h(P)$ there is some $\beta<\mathbf h(P)$ such that every element of $P_{\beta}$ dominates $P_{\leq \alpha}:=\bigcup_{\gamma\leq\alpha} P_{\gamma}$ (indeed, for every $\beta'>\beta$, every element of $P_{\beta'}$ dominates some element of $P_{\beta}$, hence dominates $P_{\leq \alpha}$). Supposing that this does not holds, then there is some $\alpha$ such that for every $\beta>\alpha$ there is some  $x_{\beta}\in P_{\beta}$ which does not dominates $P_{\leq \alpha}$, hence there is some $\alpha_{\beta}\leq \alpha$ and some $y_\beta\in P_{\alpha_{\beta}}$ which is not majorized  by $x_{\beta}$. 
The set $P_{\leq \alpha}$ is not cofinal in $P$ hence, since $P$ is pure,  it is majorized. Let $x$ be a such an element. The set $\{x_{\beta}: \alpha<\beta< \mathbf h(P)\}$ is not majorized in $P$, hence again, since $P$ is pure, it is cofinal in $P$. Hence, there is some $x_{\beta}$ which majorizes $x$. Since $x$ majorizes $P_{\leq\alpha}$ it majorizes $y_{\beta}$ hence $x_{\beta}$ majorizes $y_{\beta}$, which is impossible. The conclusion follows.
\end{proof}

\begin{lemma} \label{lem:uniquewqo} Let $P$ be  a  poset, $h: P \rightarrow K$ be an order preserving map. Suppose that $h$ is onto and $K$ has no largest element. Then  $P$ is $h$-weakly uniform iff  $P$ is $h$-minimal and, for every $\alpha \in K$, $P_{\alpha}:= h^{-1}(\alpha)$ is majorized.
If, furthermore, $P$ is wqo and $h$ is the height function $\mathbf h_P$ then these conditions amounts to the fact that 
 every chain $C$ in $P$ with order type $\mathbf h(P)$ is cofinal in $P$.
\end{lemma}
\begin{proof}
Suppose that $P$ is $h$-weakly uniform. Let $A$ be a proper initial segment of $P$.  Pick $x\in P\setminus A$. Let $\alpha:=h(x)$. If  $h(A)$ is  cofinal in $K=  h(P)$ there is some $y\in A$ with $ h(y)\geq \varphi (\alpha)$.  This element $y$ dominates all members of $P_{\alpha}$, in particular it dominates $x$, which is impossible since $x\not \in A$. Hence $P$ is $h$-minimal. Trivially, for every $\alpha\in K$, $P_{\alpha}$ is majorized by every element $y$ such that $\varphi(h(y))\geq \alpha$. 
%
%
%
%
  
 For the converse, note that in order to prove that $P$ is $h$-weakly uniform, it suffices to prove that for every $\alpha\in K$ there is some $\beta \in K$ such that every element of $P_{>\beta}$ dominates $P_{\alpha}$ and choose $\beta$ for $\varphi(\alpha)$. Supposing that this does not holds, there is some $\alpha$ such that for every $\beta>\alpha$ there is some  $x_{\beta}\in P_{>\beta}$ which does not dominates $; P_{\alpha}$, hence there some $y_\beta\in P_{\alpha}$ which is not majorized  by $x_{\beta}$.  According to our hypothesis, $P_{\alpha}$ is majorized by some $x\in P$; since $K$ has no largest element and $P$ is $h$-minimal, the set $\{x_{\beta}: \alpha<\beta\in K\}$ is cofinal in  $P$; hence $x$ is majorized by some $x_{\beta}$. This $x_{\beta}$ majorizes $y_{\alpha}$, which is impossible. 

Suppose that $P$ is w.q.o. Suppose that $P$ is $\mathbf {h}_P$-minimal.  Let $C\subseteq P$ be  a chain with order type $\kappa:= \mathbf{h}(P)$. Then $\mathbf h_{P}(C)= \kappa$  hence $\mathbf {h}(C)$ is cofinal in $\mathbf h(P)$. Since $P$ is $\mathbf {h}_P$-minimal, $C$ is cofinal in  $P$. Conversely, let $A\subseteq P$ such that $\mathbf h_{P}(A)$ is cofinal in $\mathbf h(P)$. According to K\"onig's Lemma (cf.Theorem \ref {lem:konig}), $\downarrow A$ contains a chain $C$ with order type $\mathbf h(P)$. According to $(iv)$, $C$ is cofinal in $P$. Since $C\subseteq \downarrow A$, $A$ is cofinal in $P$. Finally, these conditions imply that every level $P_{\alpha}$ is majorized. Indeed, an element $x$ of $P$ majorizes  $P_{\alpha}$ iff $x\not \in  A_{\alpha}:=\bigcup_{y\in P_{\alpha}} P\setminus \uparrow y$. Since $P_{\alpha}$ is finite, $A_{\alpha}$  is not cofinal in $P$,  hence $P\setminus A_{\alpha}\not = \emptyset$. 
\end{proof}

Let $\kappa$ be an ordinal and $\varphi: \kappa\rightarrow \kappa$ be an order-preserving and extensive map. Denote by $\varphi^{(n)}$ the $n$-th iterate of $\varphi$.  Set $A_{0}:= \downarrow \{\varphi^{(n)}(0): n<\omega\}$ and set $A_{\alpha}:= \downarrow \{\varphi^{(n)}(\delta_{\alpha}): n<\omega\}$ where $\delta_{\alpha}$ is the least element of $\kappa\setminus \bigcup_{\beta<\alpha} A_{\beta}$. Let $\mathbf h(\varphi)$ be the least $\alpha$ such that $A_{\alpha}= \emptyset$ (equivalently $\kappa= \bigcup_{\beta<\alpha} A_{\beta}$). Clearly each $A_{\alpha}$ is an interval of $\kappa$ preserved under $\varphi$. 
\begin{lemma} Let $P$ be a well founded poset and  $\kappa:= \mathbf h(P)$. Then 
$P$ is \mbox{$\mathbf h_P$-uniform} iff $P$ is a lexicographical sum $\sum_{\alpha<\kappa }Q_{\alpha}$ where each $Q_{\alpha}$ is \mbox{$\mathbf h_{Q_{\alpha}}$-uniform}.
\end{lemma} 
\begin{proof}
Let $\varphi: \kappa \rightarrow \kappa$ be an order-preserving and extensive map  $\varphi$ witnessing that   $P$ is ${\bf{h}}_P$-uniform. For each $\alpha< \mathbf h(\varphi)$,  let $Q_{\alpha}$ be the restriction of $P$ to $\mathbf{h}^{-1}_P(A_{\alpha})$. Then $P$ is the lexicographic sum $\sum _{\alpha<\mathbf h(\varphi)}Q_{\alpha}$ and each $Q_{\alpha}$ is \mbox{$\mathbf h_{Q_{\alpha}}$-uniform}. 
The fact that every element of $Q_{\alpha}$ is dominated by every element of $Q_{\alpha'}$ for $\alpha<\alpha'$ follows from the uniformity of $\varphi$. Hence $P$ is the lexicographical sum of the $Q_{\alpha}$'s. 
We have $\mathbf h_P(x)= \alpha+ \mathbf h_{Q_{\alpha}}(x)$ for every $x\in Q_{\alpha}$. Hence, $\varphi_{\alpha}$ defined by setting  $\varphi_{\alpha}(\nu):= \varphi(\alpha+\nu)$ witnesses that $Q_{\alpha}$ is $\mathbf h_{Q_{\alpha}}$-uniform.  
\end{proof}

This result yields an other proof of Theorem \ref{theo:maintheo} in the case of well founded posets.

\subsection{Spectrum and uniformity}\label{subsection:spectrum}  The \emph{spectrum} of a poset $P$ is the set $\Spec(P)$ of order types of its linear extensions (this notion was  introduced in  \cite {assous-pouzet-Nov86}). If $P$ is finite, $\vert P \vert=n$,  then all linear extensions of $P$ are isomorphic to the $n$-element chain, hence $\vert\Spec(P)\vert=1$.  The cardinality of the spectrum of an infinite antichain of size $\kappa$ is the number of isomorphic types of chains of cardinality $\kappa$. This number  is $2^{\kappa}$;  consequently, $\vert \Spec (P)\vert \leq 2^{\vert P\vert} $ for every infinite poset $P$.  As shown in \cite {assous-pouzet-Nov86}, the equality holds if   $P$ contains an  antichain of cardinality $\vert P\vert$. 
According to de Jongh-Parikh's theorem, if $P$ is w.q.o., the  cardinality of its spectrum is at most $\vert P\vert$. If in addition $P$ is J\'onsson then,  according to  Proposition \ref{prop:uniqueness},  its spectrum reduces  to a single element. Hence, \emph{among  w.q.o. posets, those whose spectrum reduces to a single element generalize J\'onsson posets}. As we will see below, some properties  of J\'onsson posets, as Theorem \ref{theo:maintheo},  extend to w.q.o.'s.

A well founded poset $P$ is {\it level-finite} if  $P_\alpha$ is finite for every ordinal $\alpha$. Let us recall the following version of K\"onig's Lemma.

\begin{theorem}\label{lem:konig} Every well founded and level-finite poset $P$ contains a chain  which intersects each level. Consequently, the supremum of the lengths of chains contained in  $P$ is attained and is equal to $\mathbf h(P)$.
\end{theorem}

From this result,  it follows immediately that a  J\'onsson poset $P$ of cardinality $\kappa$ which is level-finite has a cofinal chain of order type $cf(\kappa)$ (Theorem \ref{labtheo:cof}). Indeed, 
let $C$ be a chain going throught all the levels. Since $P$ is infinite,   $C$ has cardinality $\kappa$. Since $P$ is J\'onsson, $C$ is cofinal in $P$ and, as a chain,  has order type $\kappa$. \\

	Let $\alpha$ be an ordinal; as it is well-known there is unique pair of ordinals $\beta,r$   such that $\alpha=\omega.\beta+r$ and $r<\omega$.
%
%
 The ordinal $\omega.\beta$,  denoted by $\ell(\alpha)$, is the \textit{limit part} \index{limit part}of $\alpha$, the ordinal $r$, denoted $\alpha \textit{ mod }\omega$, is the \textit{remainder}\index{remainder}.  
%

Let $P$ be a well founded poset. Let $\alpha:= \mathbf h(P)$. Set \mbox{$\ell(P):= \{x\in P: \mathbf h_P(x) <\ell(\alpha)\}$} and $\res (P):= \{x\in P:  h(x)\geq \ell (\alpha)\}$.

\begin{lemma}\label{lem:spectsum} If $P$ is well founded and level finite then $\ell(\mathbf h(P))+\vert \res (P)\vert$ is   the least order type among the linear extensions of $P$.
\end{lemma}
\begin{proof}
The lexicographical sum $\sum_{\alpha<\mathbf {h}(P)} \overline {P}_{\alpha}$, where $\overline {P}_{\alpha}$ is a linear order on $P_{\alpha}$, is a linear extension of $P$. Its type is $\ell(\mathbf {h}(P))+\vert \res (P)\vert$. If $\overline{P}$ is a linear extension of $P$, let $a$ be the least element of $\res(P)$ in that linear extension. This element  must be minimal in $\res(P)$, hence $\mathbf {h}_P (a)= \ell (\mathbf {h}(P))$. Due to Theorem \ref{lem:konig}, the initial segment $\downarrow a$ of $P$ contains a chain with order type $\ell(\mathbf {h}(P))$, Hence $\overline P$ 
contains a chain with  order type $\ell(\mathbf h(P))+\vert \res (P)\vert$. 
\end{proof}

$\bullet$ \emph{A  well founded poset $P$ with $\vert Spec(P)\vert <2^{\aleph_0}$ is w.q.o.} Indeed, if $P$ has an infinite antichain  it has a countable one, say $A$. Let \mbox{$A^{-}:= \{x\in P: x< a\;  \text{ for some }\;  a\in A\}$} and $A^{+}= P\setminus (A\cup A^{+})$. Since $P$ is well founded, $A^{-}$ and $A^{+}$ are well founded too, hence they have a well-ordered linear extension, say $\overline{A^{-}}$ and $\overline {A^{+}}$. Each linear order $\overline A$ on $A$ yields the linear extension $\overline{A^{-}}+ \overline A+\overline {A^{+}}$.  To get $2^{\aleph_{0}}$ distinct linear extensions, select linear orders on $A$ of the form $Q_1+ B+Q_2$ where $Q_1$ and $Q_2$ are isomorphic to the chain $\Q$ of rational numbers and $B$ is a  countable scattered chain (i.e. does not contain a copy of the chain $\Q$). From the fact that there are $2^{\aleph_{0}}$ such $B$ which are pairwise non isomorphic the conclusion follows.

Let  $P$ be a w.q.o. and let $\mathbf m(P)$ be  the least order type among the linear extensions of $P$. By definition $\mathbf m(P)\leq \mathbf o(P)$; the equality hold iff $\vert \Spec(P)\vert =1$.   
As shown below, the description of wqo's $P$ which have a unique order type of linear extension reduces to those for which the height is a limit ordinal

\begin{lemma} Let $P$ be  a wqo. If  $\mathbf h(P)$ is not a limit ordinal then $\vert \Spec(P)\vert =1$ iff $P$ is the lexicographical sum $\ell(P)+ \res(P)$ and $\Spec (\ell(P))=1$. 
\end{lemma}
\begin{proof}If $P$ is the lexicographical sum $\ell(P)+ \res(P)$ then every linear extension $\overline P$ of $P$ is the sum
of a linear extension of $\ell(P)$ and a linear extension of $\res(P)$. Since $\res(P)$ is finite, its  linear  extensions  are isomorphic to a $m$-element chain where $m = \vert \res(P)\vert$, hence $\Spec(P)=\Spec(\ell(P))$.  Suppose that $\Spec(P)= 1$. Then, according to Lemma \ref{lem:spectsum}, the linear extensions of $P$ have order type $\ell(\mathbf h(P))+\vert \res (P)\vert$. Let $a$ be a minimal element of $\res(P)$. We claim that $a$ dominates $\ell(P)$. Otherwise, $\ell(P)\setminus \downarrow a$ is non-empty, hence  \mbox{$(\downarrow a)\cap \ell(P)+ \ell(P)\setminus a+ \res(P)$} is an extension of $P$ whose  every linear extension has type larger that  $\ell(\mathbf h(P))+\vert \res (P)\vert$. A contradiction. Since every minimal element of $\res(P)$ dominates $\ell(P)$,  $P$ is the lexicographical sum $\ell(P)+ \res(P)$. 
\end{proof}


\begin{proposition} Let $P$ be a w.q.o.  such that  $\nu:=\mathbf h(P)$ is a limit  ordinal. If 
$\vert \Spec(P)\vert=1$ then $P$ is $\mathbf h_P$-weakly uniform.  \end{proposition}
\begin{proof}
If $P$ is not $\mathbf h_P$-weakly uniform  then by Lemma \ref{lem:uniquewqo},  $P$ is not $\mathbf h_P$-minimal, hence there is some $A$ not cofinal in $P$ with $\mathbf h_P (A)$ cofinal in $\nu$. Let $\overline A:= \downarrow A$. Then $\overline A\not = P$.  Since $\mathbf h_P{\restriction {\overline A}}= \mathbf {h}_{\overline A}$ we have  $\mathbf h_P(\overline A)= \nu$. Hence $$\mathbf o(P)>\mathbf o(\overline A)\geq \mathbf h_P(\overline A)= \mathbf h(P)= \mathbf m(P). $$  This yields  $\vert \Spec(P)\vert\not = 1$. 
\end{proof}

The converse of this property does not hold. Let $\alpha$ be an ordinal, denote by  $C_{\alpha}$ any chain of order type $\alpha$. Let $D_{\alpha, \beta}:= (C_{\alpha} \oplus C_{\alpha}) + C_{\beta}$ be the lexicographic sum of the direct sum of two copies of $C_{\alpha}$ with a copy of $C_{\beta}$. Then $\mathbf h(D_{\alpha, \beta})= \alpha+\beta$ and $D_{\alpha, \beta}$ is $\mathbf {h}_{D_{\alpha, \beta}}$-uniform. Also $\mathbf o(D_{\alpha, \beta})= (\alpha\oplus \alpha)+ \beta$ (the symbol $\oplus$ denotes the Heissenberg sum). If $\alpha=\beta$ then this quantity is strictly greater than $\mathbf m(P)$ hence $\vert \Spec(P)\vert>1$. If $\beta$ is indecomposable  and $\alpha< \beta$  then $\mathbf o(P)= \beta$ and  $\vert \Spec(P)\vert=1$.

Here is  an  example   of well founded poset $Q$   which is   $\mathbf h_Q$-weakly uniform but not $\mathbf h_Q$-uniform.  Let $Q$ be disjoint union of a chain $C_{\delta}$ of order type $\delta:= \omega_{\beta}$, and a  chain $C_{\gamma}$ of order type $\gamma:= \delta^{\delta}:= \sum_{\alpha< \delta}\delta^{\alpha}$.  Select in $C_{\gamma}$ a strictly increasing and cofinal sequence $(c_{\alpha})_{\alpha< \delta}$ with $c_{0}\geq \omega$. And for $x, y\in Q$, set $x<y$ if either $x, y$ are ordered according to one of the  two chains, or $x:= \alpha \in C_{\delta}$  and $y\in C_{\gamma}$ with $y\geq c_{\alpha}$. Then $\mathbf h(Q):= \gamma= \mathbf o(Q)$ and $Q$ is $\mathbf{h}_Q$-weakly uniform but not $\mathbf{h}_Q$-uniform.


\section{A conjecture on w.q.o.'s}
C.St.J.A.~Nash-Williams introduced the  notion of better-quasi-ordering (b.q.o.),  a strengthening of the  notion of w.q.o. (cf. \cite{nashwilliams1}, \cite{nashwilliams2}). The operational definition is not intuitive;  since we are not going to  prove results about b.q.o.s,  an intuitive definition is enough. Let $P$ be a poset and $P^{<\omega_1}$,  the set of maps $f: \alpha \rightarrow P$, where $\alpha$ is any countable ordinal. If $f$ and $g$ are two such maps, we set $f\leq g$ if there is a one-to-one preserving map $h$ from the domain $\alpha$ into the domain $\beta $ of $g$ such that $f(\gamma)\leq g( h(\gamma))$ for all $\gamma< \alpha$. This relation is  a	quasi-order; the poset $P$ is a \emph{better quasi-order} if $P^{<\omega_1}$ is w.q.o. B.q.o.s are w.q.o.s. As  w.q.o.s, finite sets and  well-ordered sets are b.q.o. (do not try to prove it with this definition), finite unions, finite products, subsets and images of b.q.o. are b.q.o. But, contrarily to w.q.o.,  if $P$ is b.q.o. then $\mathbf{I} (P)$ is bqo.  

In \cite{abraham-bonnet-kubis} was made the following:

\begin{conjecture}
Every w.q.o. is a countable union of b.q.o.
\end{conjecture} 
In order to attack this conjecture, the second author asked more: is every w.q.o. a countable union of posets, each one being a strengthening of some  finite dimensional poset? 

To prove the validity of these conjectures it is natural to use  induction. Induction can be on the ordinal type  or, perhaps better,  on the \emph{ordinal rank} of antichains of w.q.o'.s: if $P$ is a poset, order the set $\mathcal A(P)$ of antichains by reverse of the inclusion,  if $P$ has no infinite antichain, $\mathcal A(P)$ is well founded, so the empty antichain $\emptyset$ has a height  in this  poset, we denote it by $rank(\mathcal A(P))$. It was shown by Abraham \cite{abraham} that $rank(\mathcal A(P))< \omega_1^2$ iff $P$ is a countable union of chains. So for this value of the rank,  the conjecture holds. In particular it holds for w.q.o. posets for which $o(P)<\omega_1^2$ simply because   $rank(\mathcal A(P))\leq o(P)$. In fact, if $o(P)<\omega_1^2$,  Abraham's result is immediate, indeed $P$ decomposes into a countable union  of posets $P_{\alpha}$, each of ordinal length a most $\omega_1$. Hence  $P$ decomposes into a countable union of J\'onsson posets. According to Corollary \ref{cor:width}  each one is a contable union of chains hence $P$ is a countable union of chains. The first instance of poset  for which the conjecture poses problem is a w.q.o. which is an uncountable union of chains of type $\omega_1$ and not less.


\begin{thebibliography}{99}

\bibitem{abraham} U.~Abraham, A note on Dilworth's theorem in the infinite case,  Order {\bf 4} (1987), no. 2, 107--125. 
\bibitem{abraham-bonnet-kubis} U.~Abraham, R.~Bonnet, W.~Kubis,  Poset algebras over well quasi-ordered posets,  Algebra Universalis {\bf 58} (2008), no. 3, 263--286. 


 \bibitem {allouche}J-P.~Allouche, J.~Shallit,  Automatic sequences. Theory, applications, generalizations. Cambridge University Press, Cambridge, 2003. xvi+571.

\bibitem{assous-pouzet-Nov86}  R.~Assous, M.~Pouzet, Spectre des extensions lin\'eaires d'un belordre,   Rapport LAOA, Novembre 1986, Lyon.

\bibitem{brt} K.~P. Bogart, I.~Rabinovich and W. T.~Trotter Jr., A bound on the dimension of interval orders,  Journal of Combinatorial Theory, Series A \textbf{21} (1976), 319--328.

\bibitem {candeal} J.C.~Candeal, E. Indur\'ain, Semiorders and thresholds of utility discrimination: solving the Scott-Suppes representability problem, J. Math. Psych. {\bf 54} (2010), no. 6, 485--490.

\bibitem{chudnovski} M.~Chudnovsky, R.~Kim, S.~Oum, P.~Seymour,  Unavoidable induced subgraphs in large graphs with no homogeneous sets, J. Combin. Theory Ser. B {\bf 118} (2016), 1--12. 
\bibitem{dejongh-parikh}
D.H.J. de~Jongh and R.~Parikh,
 Well-partial orderings and hierarchies, 
Nederl. Akad. Wetensch. Proc. Ser. A {\bf 80}=Indag. Math. 
  {\bf39(3)} (1977), pp.~195--207.

\bibitem{delhomme-pouzet} C.~Delhomm\'e, M.~Pouzet, Length of an intersection, Mathematical Logic Quarterly 1-13(2017), DOI 10.1002/malq.201500067,   arXiv:1510.00596

\bibitem{fi-book} P. C.~Fishburn. Interval orders and interval graphs.  John Willey \& Sons, 1985.

\bibitem{fi} P. C.~Fishburn, Intransitive indifference with unequal indifference intervals, J. Math. Psych. \textbf{7} (1970) 144--149.

\bibitem {fogg} Pytheas~Fogg.
\newblock { Substitutions in Dynamics, Arithmetics and Combinatorics.}
\newblock { V.Berth\'e, S.Ferenczi, C.Mauduit, A.Siegel (Eds), Springer 2012.}

\bibitem{fraisse} R.~Fra\"{\i}ss\'e.  {Theory of relations}. Revised edition. With an appendix by Norbert Sauer. Studies in Logic and the Foundations of Mathematics, 145. North-Holland Publishing Co., Amsterdam, 2000. ii+451 pp.  

\bibitem{jech} T.~Jech. { Set Theory. 3rd millennium edn}. Springer Monographs in Mathematics. Springer, New York (2002). 

\bibitem{kearnes}K.~Kearnes, G.~Oman,  J\'onsson posets and unary J\'onsson algebras,  Algebra Universalis {\bf 69} (2013), no. 2, 101--112.
 \bibitem{kim}  R.~Kim,  Unavoidable subtournaments in large tournaments with no homogeneous sets, SIAM J. Discrete Math. {\bf 31} (2017), no. 2, 714--725.

\bibitem{kok}J.~Kok, N. K.~Sudev, K.P.~Chithra, U.~Mary, Jaco-Type Graphs and Black Energy Dissipation, 
Advances in Pure and Applied Mathematics,  {\bf 8} (2017) Issue 2,  141-152,   arXiv:1607.00472, 12 Oct. 2016. 

\bibitem{lothaire}  M. Lothaire,  {Finite and Infinite Words}. Algebraic Combinatorics on Words. Cambridge University Press. 2002.
\bibitem{luce} R. D.~Luce, Semiorders and a theory of utility discrimination,   Econometrica \textbf{24} (1956), 178--191.

\bibitem{malliaris-terry} M.~Malliaris, C.~Terry, On unavoidable induced subgraphs in large prime graphs, 14p. arXiv:1511.02544v1, 9 Nov. 2015. 
\bibitem{milner-pouzet} E.C.~Milner, M.~Pouzet,  On the cofinality of partially ordered sets,  Ordered sets I.Rival ed.(Banff, Alta., 1981), pp. 279--298, NATO Adv. Study Inst. Ser. C: Math. Phys. Sci., {\bf 83}, Reidel, Dordrecht-Boston, Mass., 1982. 


\bibitem{nashwilliams1}
C.St.J.A.~Nash-Williams, 
\newblock On well-quasi-ordering infinite trees, 
\newblock {\em Proc., Phil, Soc.}, {\bf 61} (1965), 697--720.

\bibitem{nashwilliams2}
C.St.J.A.~Nash-Williams,
\newblock On well-quasi-ordering transfinite sequences.
\newblock {\em Proc., Phil, Soc.}, {\bf 61} (1965),  33--39.

\bibitem{oudrarthese}
D.~Oudrar,
Sur l'\'enum\'eration de structures discr\`etes. Une approche par la th\'eorie des relations. 
 {\em Th\`ese de Doctorat. Universit\'e des sciences et de la technologie Houari Boumediene, U.S.T.H.B., Alger (28 Septembre 2015) 248 pages.}
 {Available at arXiv:1604.05839[math.CO].}


\bibitem{pirlotvincke} M.~Pirlot and P.~Vincke, {Semiorders: Properties, Representations, Applications.} Volume 36 of Theory and Decision Library Series B, Springer Science \& Business Media, 1997. 

\bibitem{pouzet-these}M.~Pouzet,
 {\em Sur la th\'eorie des relations}.
 Th\`ese d'\'etat, Universit\'e Claude-Bernard, Lyon 1, pp.~78--85, 1978.

\bibitem{pouzet-sauer}M.~Pouzet, N.~Sauer,  From well-quasi-ordered sets to better-quasi-ordered sets,  Electron. J. Combin. {\bf 13}, (2006), no. 1, Research Paper 101, 27 pp. (electronic). 

\bibitem{pouzet-zaguia} M.~Pouzet, I.~Zaguia, On minimal prime graphs and posets,  Order {\bf 26} (2009), no. 4, 357--375. 
\bibitem{pouzet-zaguia2} M.~Pouzet, I.~Zaguia, Interval orders, semiorders and ordered groups, 26 pages,  arXiv: 1706.03276v1.
 
 \bibitem{rabinovitch} I.~Rabinovitch,The dimension of semiorders. J. Comb. Theory Ser A \textbf{25} (1978), 50--61.


\bibitem{scottsuppes} D.~Scott and P.~Suppes, Foundational aspects of theories of measurement,  The Journal of Symbolic Logic \textbf{23} (1958), 113--128.

\bibitem{wiener} N.~Wiener, A Contribution to the Theory of Relative Position,   Proc. Cambridge Philos. Soc. \textbf{17} (1914), 441--449.

\bibitem{wolk} E.S.~Wolk, Partially well ordered sets and partial ordinals, 
Fund. Math. {\bf 60} (1967) 175--186.

\end{thebibliography}
\end{document}